\nonstopmode
\documentclass[reqno,10pt]{amsart}
\usepackage{latexsym}
\usepackage{fancyhdr}
\usepackage{amsmath, amssymb}
\usepackage[ansinew]{inputenc}
\usepackage[all]{xy}
\usepackage{pdflscape}
\usepackage{longtable}
\usepackage{rotating}
\usepackage{verbatim}
\usepackage{hyperref}
\usepackage{subfigure}
\usepackage{mathrsfs}
\usepackage{cleveref}
\usepackage{mdwlist}
\usepackage{color}

\newcounter{mnotecount}[section]

\newcommand{\rmnote}[1]{}

\DeclareFontFamily{U}{mathb}{\hyphenchar\font45}
\DeclareFontShape{U}{mathb}{m}{n}{
      <5> <6> <7> <8> <9> <10> gen * mathb
      <10.95> mathb10 <12> <14.4> <17.28> <20.74> <24.88> mathb12
      }{}
\DeclareSymbolFont{mathb}{U}{mathb}{m}{n}
\DeclareFontSubstitution{U}{mathb}{m}{n}

\let\dot\relax
\DeclareMathAccent{\dot}{0}{mathb}{"39}
\let\ddot\relax
\DeclareMathAccent{\ddot}{0}{mathb}{"3A}
\let\dddot\relax
\DeclareMathAccent{\dddot}{0}{mathb}{"3B}
\let\ddddot\relax
\DeclareMathAccent{\ddddot}{0}{mathb}{"3C}

\theoremstyle{plain}
\newtheorem*{theorem*}{Theorem}
\newtheorem{theorem}{Theorem}
\newtheorem*{lemma*}{Lemma}
\newtheorem{lemma}[theorem]{Lemma}
\newtheorem*{proposition*}{Proposition}
\newtheorem{proposition}[theorem]{Proposition}
\newtheorem*{corollary*}{Corollary}
\newtheorem{corollary}[theorem]{Corollary}
\newtheorem*{claim*}{Claim}

\newtheorem*{conjecture*}{Conjecture}

\newtheorem*{question*}{Question}
\theoremstyle{definition}
\newtheorem*{definition*}{Definition}

\newtheorem*{example*}{Example}

\newtheorem*{algorithm*}{Algorithm}
\newtheorem*{remark*}{Remark}
\newtheorem*{remarks*}{Remarks}

\newtheorem*{convention*}{Convention}

\newcommand{\al}{\alpha}
\newcommand{\be}{\beta}

\newcommand{\et}{\eta}

\newcommand{\ka}{\kappa}
\newcommand{\la}{\lambda}

\newcommand{\rh}{\rho}

\newcommand{\si}{\sigma}

\newcommand{\vh}{\varphi}

\newcommand{\om}{\omega}

\newcommand{\Ga}{\Gamma}

\newcommand{\N}{\mathbb{N}}

\newcommand{\R}{\mathbb{R}}

\newcommand{\cB}{\mathcal{B}}

\newcommand{\cE}{\mathcal{E}}

\newcommand{\fM}{\mathfrak{M}}
\newcommand{\fN}{\mathfrak{N}}

\newcommand{\fS}{\mathfrak{S}}

\newcommand{\fV}{\mathfrak{V}}
\newcommand{\fW}{\mathfrak{W}}

\newcommand{\p}{\partial}

\renewcommand{\o}{\circ}

\newcommand{\on}{\operatorname}

\newcommand{\sr}[1]%
{\ifmmode{}^\dagger\else${}^\dagger$\fi\ifvmode
\vbox to 0pt{\vss
 \hbox to 0pt{\hskip\hsize\hskip1em
 \vbox{\hsize3cm\raggedright\pretolerance10000
 \noindent #1\hfill}\hss}\vss}\else
 \vadjust{\vbox to0pt{\vss%
 \hbox to 0pt{\hskip\hsize\hskip1em%
 \vbox{\hsize3cm\raggedright\pretolerance10000%
 \noindent #1\hfill}\hss}\vss}}\fi%
}

\newcommand{\A}{\;\forall}
\newcommand{\E}{\;\exists}

\newcommand{\ol}{\overline}
\newcommand{\ul}{\underline}

\title[On the extension of Whitney ultrajets, II]
{On the extension of Whitney ultrajets, II}

\author[A.~Rainer]{Armin Rainer}

\address{A.~Rainer: 
Fakult\"at f\"ur Mathematik, Universit\"at Wien, 
Oskar-Morgenstern-Platz~1, A-1090 Wien, Austria}
\email{armin.rainer@univie.ac.at}

\author[G.~Schindl]{Gerhard Schindl}

\address{G.~Schindl: Departamento de \'Algebra, An\'alisis Matem\'atico, Geometr\'{\i}a y Topolog\'{\i}a,
Facultad de Ciencias,
Universidad de Valladolid,
Paseo de Bel\'en 7,
47011 Valladolid, Spain}
\email{gerhard.schindl@univie.ac.at}

\begin{document}

\begin{abstract}
We characterize the validity of the Whitney extension theorem in the ultradifferentiable Roumieu setting 
with controlled loss of regularity. 
Specifically, we show that in the main Theorem 1.3 of \cite{RainerSchindl17} condition (1.3) can be dropped. 
Moreover,
we 
clarify some questions that remained open in \cite{RainerSchindl17}.  
\end{abstract}

\thanks{The first author was supported by the Austrian Science Fund (FWF) Project P 26735-N25.
The second author was supported by FWF-Project J 3948-N35; 
within this project, he is an external researcher at the Universidad de Valladolid (Spain) for the 
period October 2016 -- September 2018}
\date{\today}

\maketitle

\section{Introduction}

The main goal of this paper is to prove:

\begin{theorem} \label{mainadd}
Let $\om$ be a non-quasianalytic concave weight function. 
Let $\si$ be a weight function satisfying $\si(t) = o(t)$ as $t \to \infty$. 
Then the following conditions are equivalent:
\begin{enumerate}
    \item[(i)] For every compact $E \subseteq \R^n$ we have 
        $j^\infty_E(\cB^{\{\om\}}(\R^n)) \supseteq \cB^{\{\si\}}(E)$, 
        where $j_E^\infty$ assigns to each $f \in C^\infty(\R^n)$ its infinite jet $(f^{(\al)}|_E)_{\al \in \N^n}$ on $E$. 
    \item[(ii)] There is $C>0$ such that $\int_{1}^\infty \frac{\om(tu)}{u^2}\,du \le C\si(t) + C$ for all $t>0$.
  \end{enumerate}
\end{theorem}

(Here $\cB^{\{\om\}}$ denotes the Roumieu class defined by the weight function $\om$; we use the symbol $\cB$ to emphasize that the 
defining estimates are global, cf.\ \cite[2.2 and 2.6]{RainerSchindl17}.)
It means that Theorem 1.3 of \cite{RainerSchindl17} holds without the assumption (1.3) that the associated weight matrix $\fS$ of $\si$
satisfies
\begin{equation}
   \label{intro:good}
    \A S \in \fS \E T \in \fS \E C \ge 1 \A 1\le j \le k : \frac{S_j}{jS_{j-1}} \le C\, \frac{T_k}{k T_{k-1}}. 
\end{equation}
\Cref{mainadd} is proved in \Cref{sec:proof}. 
In \Cref{sec:questions} we clarify some questions that remained open in \cite{RainerSchindl17} and 
obtain several characterizations of concave weight functions. For an overview of the background of \Cref{mainadd} we 
refer to the introduction in \cite{RainerSchindl17}.  
We use the notation and the definitions of said paper; 
the concept of \emph{weight matrices} is recalled in the appendix at the end of this paper. 

Note that in the special case that $\om$ and $\si$ coincide we recover the result of \cite{BBMT91}:

\begin{corollary}
  \label{thm:preservingclass}
  Let $\om$ be a weight function. The following conditions are equivalent:
  \begin{enumerate}
    \item[(i')] For every compact $E \subseteq \R^n$ we have $j^\infty_E(\cB^{\{\om\}}(\R^n)) = \cB^{\{\om\}}(E)$.
    \item[(ii')] There is $C>0$ such that $\int_{1}^\infty \frac{\om(tu)}{u^2}\,du \le C\om(t) + C$ for all $t>0$.
  \end{enumerate}
\end{corollary}

Indeed,
if $\om$ satisfies (ii') then it is non-quasianalytic, equivalent to a concave weight function 
\cite[Proposition 1.3]{MeiseTaylor88}, and $\om(t) = o(t)$ as $t \to \infty$ \cite[Remark 3.20]{BonetMeiseTaylor92}. 
That (ii') is a necessary condition for (i') is well-known. 
Note that also (i') implies that $\om$ is non-quasianalytic. 
 Indeed, if $\om$ is quasianalytic, 
 then the Borel map $j^\infty_{\{0\}} : \cB^{\{\om\}}(\R^n) \to \cB^{\{\om\}}(\{0\})$ is never surjective.
 For $t \ne O(\om(t))$ as $t \to \infty$ this follows from \cite{RainerSchindl15}, 
 for $t = O(\om(t))$ as $t \to \infty$ consider e.g.\ the formal series $\sum_{k=0}^\infty x^k$ which converges to the  
 unbounded real analytic function $1/(1-x)$ function for $|x|<1$.

\section{Proof of \texorpdfstring{\Cref{mainadd}}{Theorem 1}} \label{sec:proof}

\subsection*{Preparations} 
First we recall a few definitions and facts. 
Let $m =(m_k)$ be a positive sequence satisfying $m_0 = 1$ and $m_k^{1/k} \to \infty$.
The \emph{log-convex minorant} of $m$ is given by  
\[
  \ul m_k := \sup_{t>0} \frac{t^k}{\exp(\om_m(t))}, \quad k \in \N,
\]
where 
\[
  \om_m (t) :=  \sup_{k \in \N}  \log \Big(\frac{t^k}{m_k}\Big), \quad t>0. 
\] 
The function $\om_m$ is increasing, convex in $\log t$, and zero for sufficiently small $t>0$. 
Related is the function $h_m(t) := \inf_{k \in \N} m_k t^k$, for $t>0$, and $h_m(0):=0$. 
It is increasing, continuous, positive for $t>0$, and equals $1$ for large $t$.

Let $m=(m_k)$ be a positive \emph{log-convex} sequence (i.e., $m = \ul m$) such that $m_0 =1$ and 
$m_k^{1/k} \to \infty$. 
Then the functions $\ol \Ga_m$ and $\ul \Ga_m$ defined in \cite[Definition 3.1]{RainerSchindl17} coincide,
we simply write $\Ga_m$ in this case; note that log-convexity and $m_k^{1/k} \to \infty$ imply $m_{k}/m_{k-1} \to \infty$. Thus
\begin{equation} \label{eq:Gamma}
  \Ga_m(t) = \min\{ k : h_m(t) = m_k t^k\} = \min \Big\{k : \frac{m_{k+1}}{m_k} \ge \frac1t\Big\}, \quad t>0. 
\end{equation}
By \cite[Lemma 3.2]{RainerSchindl17}, $\Ga_m$ is decreasing, tending to $\infty$ as $t \to 0$, and 
\begin{equation} \label{eq:Gamma1}
   k \mapsto m_k t^k \text{ is decreasing for } k \le \Ga_m(t).
\end{equation}

Recall that with every weight function $\si$ (always understood as defined in \cite[Section 2.1]{RainerSchindl17}) 
is associated a weight matrix $\fS = \{S^\xi\}_{\xi>0}$, where
\[
  S^\xi_k := \exp\big(\tfrac1\xi \vh^*(\xi k\big)), \quad \text{ (here $\vh = \si \o \exp$ and $\vh^*$ is its Young conjugate), }   
\]
such that 
$\cB^{\{\si\}} = \cB^{\{\fS\}}$ and $\cB^{(\si)} = \cB^{(\fS)}$ algebraically and topologically; 
cf.\ \cite[2.5]{RainerSchindl17} and \cite{RainerSchindl12}. 
In the following we set $s^\xi_k := S^\xi_k/k!$. 

The next proposition
shows that for a weight function $\si$
which is equivalent 
to a concave weight function
and satisfies $\si(t) = o(t)$ as $t \to \infty$  we additionally have $\cB^{\{\si\}} = \cB^{\{\ul \fS\}}$ and $\cB^{(\si)} = \cB^{(\ul \fS)}$, 
where $\ul \fS = \{\ul S^\xi\}_{\xi>0}$ and
\[
    \ul S^\xi_k:= k!\, \ul s^\xi_k.
\]
In particular, $\ul \fS$ satisfies \eqref{intro:good}. We say that $\ul S^\xi$ is \emph{strongly log-convex} meaning that 
$\ul s^\xi_k=\ul S^\xi_k/k!$ is log-convex. 
(Note the abuse of notation: $\ul S^\xi$ is \emph{not} necessarily the log-convex minorant of $S^\xi$; this will cause no confusion.)
Recall that two weight functions $\om$ and $\si$ are called \emph{equivalent} 
if $\om(t)= O(\si(t))$ and $\si(t) = O (\om(t))$ as $t \to \infty$; this means that they define the same ultradifferentiable class. 

\begin{proposition} \label{prop:strongmatrix}
  Let $\si$ be a weight function satisfying $\si(t) = o(t)$ as $t \to \infty$ which is equivalent 
  to a concave weight function. For each $\xi>0$ there exist constants $A,B,C >0$ such that 
  \begin{equation} \label{eq:strong}
  		A^{-1} s^{\xi/B}_k \le \ul s^\xi_k \le s^\xi_k \le C^k \ul s^{B\xi}_k \quad \text{ for all } k \in \N. 
  \end{equation}
  Moreover, there is a constant $H\ge 1$ such that $\ul s^\xi_{j+k} \le H^{j+k} \ul s^{2\xi}_{j} \ul s^{2\xi}_{k}$,
  for all $\xi>0$ and all $j,k \in \N$, and 
  thus
  $h_{\ul s^{\xi}}(t) \le h_{\ul s^{2\xi}}(Ht)^2$, for all $\xi>0$ and all $t>0$.  
\end{proposition}

\begin{proof}
  Clearly, $\ul s^\xi \le s^\xi$. Let $\ul S^\xi_k:= k!\, \ul s^\xi_k$.
  By \cite[Lemma 3.6]{Jimenez-GarridoSanzSchindl17}, 
  $\om_{S^\xi}$ and $\om_{\ul S^\xi}$ are equivalent, in particular, there exists $C\ge 1$ such that 
  \begin{equation} \label{eq:equi}
  \om_{\ul S^\xi} \le C \om_{S^\xi} + C.
  \end{equation}
  By \cite[Lemma 2.4(3)]{RainerSchindl17} and \cite[Remark 2.5]{RainerSchindl16a}, we have
	\begin{equation*} 
  2 \om_{S^{2\xi}} \le \om_{S^\xi}, \quad \text{ for all } \xi>0.  
  \end{equation*}
  If $n$ is an integer such that $B:= 2^n \ge C$, then 
  $\om_{\ul S^\xi} \le  \om_{S^{\xi/B}} + C$ and hence
  \[
  \ul S^\xi_k = \sup_{t>0} \frac{t^k}{\exp(\om_{\ul S^\xi}(t))} \ge e^{-C} \sup_{t>0} \frac{t^k}{\exp(\om_{S^{\xi/B}}(t))} = 
  e^{-C} S^{\xi/B}_k.
  \]
  This shows the first inequality in \eqref{eq:strong}.

  By \cite[Lemma 3.13]{RainerSchindl17},  
there exists $D\ge 1$ such that for all $\xi>0$, 
  \begin{equation*} 
  2 \om_{s^{2\xi}}(t) \le \om_{s^\xi}(Dt), \quad \text{ for } t>0   
  \end{equation*}
   and therefore
  \begin{align*}
  \ul s^\xi_{2k} 
  = \sup_{t>0} \frac{(Dt)^{2k}}{\exp(\om_{s^\xi}(Dt))}
  \le D^{2k} \sup_{t>0} \frac{t^{2k}}{\exp(2\om_{s^{2\xi}}(t))} 
  = D^{2k} (\ul s^{2\xi}_{k})^2.
  \end{align*}
  Thus, by \cite[Theorem 9.5.1]{Schindl14} (which is a generalization of \cite{Matsumoto84}), 
  there exists a constant $H\ge 1$ such that 
  $\ul s^\xi_{j+k} \le H^{j+k} \ul s^{2\xi}_{j} \ul s^{2\xi}_{k}$, for all $j,k$.
  That $h_{\ul s^{\xi}}(t) \le h_{\ul s^{2\xi}}(Ht)^2$, for all $\xi>0$ and all $t>0$, follows from \cite[Lemma 3.12]{RainerSchindl17}. 
  By \cite[Proposition 3.6]{Schindl15},  
  \begin{equation*}
  2 \om_{\ul S^{2\xi}}(t) \le \om_{\ul S^\xi}(Ht), \quad \text{ for } t>0,   
  \end{equation*}
  for some (possibly different) $H\ge 1$. 
  As above, using \eqref{eq:equi}, we find $\om_{\ul S^{B \xi}}(b t) \le \om_{S^\xi}(t) + 1$ for some constant $0< b \le 1$. 
  Then 
  \begin{align*}
  \ul S^{B \xi}_k  =  \sup_{t>0} \frac{(bt)^k}{\exp(\om_{\ul S^{B\xi}}(bt))}  \ge 
  e^{-1} b^k \sup_{t>0} \frac{t^k}{\exp(\om_{S^{\xi}}(t))} = 
  e^{-1}b^k S^\xi_k.
  \end{align*}
  The last inequality of \eqref{eq:strong} follows.	   
\end{proof}

\Cref{prop:strongmatrix} alone is not enough to get rid of the assumption \eqref{intro:good}. 
It is not clear that $\ul \fS$ has the property that for all $S \in \ul \fS$ there is a $T \in \ul \fS$ 
such that $S_{2k}/S_{2k-1} \lesssim T_{k}/T_{k-1}$. Note that $\fS$ has this property 
(see \cite[Lemma 2.4(4)]{RainerSchindl17}) and it enters crucially in Lemma 3.4
and Proposition 3.7 of
\cite{RainerSchindl17}. \label{reasongood}

We deal with this problem by introducing another intimately related weight matrix $\fV := \{V^\xi\}_{\xi>0}$.  
For each $\xi>0$ we define $V^\xi_k := k!\, v^\xi_k$ by setting
\begin{equation} \label{eq:save-1}
  v^\xi_k := \min_{0 \le j \le k} \ul s^{2\xi}_{j} \ul s^{2\xi}_{k-j}, \quad k \in \N.
\end{equation}
That means that for the sequence of quotients $v^\xi_k/v^\xi_{k-1}$ we have (cf.\ \cite[Lemma 3.5]{Komatsu73}) 
\[
  \Big(\frac{v^\xi_1}{v^\xi_{0}},\frac{v^\xi_2}{v^\xi_{1}},\frac{v^\xi_3}{v^\xi_{2}},\frac{v^\xi_4}{v^\xi_{3}}, \ldots\Big) = 
  \Big(\frac{\ul s^{2\xi}_1}{\ul s^{2\xi}_{0}},\frac{\ul s^{2\xi}_1}{\ul s^{2\xi}_{0}},\frac{\ul s^{2\xi}_2}{\ul s^{2\xi}_{1}}, 
  \frac{\ul s^{2\xi}_2}{\ul s^{2\xi}_{1}},\frac{\ul s^{2\xi}_3}{\ul s^{2\xi}_{2}}, \frac{\ul s^{2\xi}_3}{\ul s^{2\xi}_{2}}, \ldots\Big). 
\]
Thus the sequence $v^\xi = (v^\xi_k)$ is log-convex and satisfies
\begin{equation} \label{eq:save}
    \frac{v^\xi_{2k-1}}{v^\xi_{2k-2}} = \frac{v^\xi_{2k}}{v^\xi_{2k-1}} = \frac{\ul s^{2\xi}_{k}}{\ul s^{2\xi}_{k-1}}, 
    \quad \text{ for all } k\ge 1.
\end{equation}
So, in view of \eqref{eq:Gamma}, 
\begin{equation} \label{eq:save2}
   2 \Ga_{\ul s^{2\xi}}(t) = \Ga_{v^\xi}(t), \quad \text{ for all } t>0. 
\end{equation}
By \Cref{prop:strongmatrix}, there is $H\ge 1$ such that for all $\xi>0$
\begin{equation} \label{eq:save1}
  \ul s^\xi_k \le H^k v^\xi_k \le H^k \ul s_k^{2\xi}, \quad \text{ for all } k \in \N.
\end{equation}
Thus, we also have 
$\cB^{\{\si\}} = \cB^{\{\fV\}}$ and $\cB^{(\si)} = \cB^{(\fV)}$ algebraically and topologically.

\subsection*{Proof of \texorpdfstring{\Cref{mainadd}}{Theorem 1}}
  The implication (i) $\Rightarrow$ (ii) follows from \cite{BonetMeiseTaylor92}. So we only prove the converse implication.  
  Condition (ii) means that the weight function 
  \begin{equation} \label{kappa}
  	\ka(t) := \int_1^\infty \frac{\om(tu)}{u^2}\, du
  \end{equation} 
  satisfies $\ka(t) = O(\si(t))$ as $t\to \infty$, i.e., $\cB^{\{\si\}} \subseteq \cB^{\{\ka\}}$. 
  Now $\ka$ is concave and $\ka(t)= o(t)$ as $t \to \infty$, see \cite[Proposition 1.3]{MeiseTaylor88}. 
  We will show that Whitney ultrajets of class $\cB^{\{\ka\}}$ admit extensions of class $\cB^{\{\om\}}$.
  Thus from now on we assume without loss of generality that $\si = \ka$ is concave. 
  Since $\om$ is increasing we have $\si=\ka \ge \om$ and hence, if $\fW = \{W^\xi\}_{\xi>0}$ denotes the weight matrix associated with $\om$,
  \begin{equation} \label{eq:order}
    \ul S^{\xi} \le S^{\xi} \le W^\xi, \quad \text{ for all } \xi>0.
  \end{equation}
  Moreover, \Cref{prop:strongmatrix} as well as \eqref{eq:save2} and \eqref{eq:save1} apply. 
  Let us now indicate the necessary changes in the proof of \cite[Theorem 1.3]{RainerSchindl17}. 
  The changes also lead to some simplifications. 
  We provide details in the hope that this contributes to a better understanding.

  $\bullet$ Every Whitney ultrajet $F=(F^\al)$ of class $\cB^{\{\si\}}$ on the compact set $E \subseteq \R^n$ 
  is an element of $\cB^{\{V^\xi\}}(E)$ for some $\xi>0$, i.e.,  
there exist $C>0$ and $\rh \ge 1$ 
such that 
\begin{gather}
  |F^\al(a)| \le C \rh^{|\al|} \,  V^\xi_{|\al|}, \quad \al \in \N^n,~ a \in E,
   \label{jets1}
  \\
  |(R^k_a F)^\al(b)| \le C \rh^{k+1} \, |\al|!\, v^\xi_{k+1}\,  |b-a|^{k+1-|\al|}, \quad k \in \N,\, |\al| \le k,~ a,b \in E.  
  \label{jets2}
\end{gather}  
   Let $p\in \N$ be fixed (and to be specified later). 
Let $\{\vh_{i,p}\}_{i\in \N}$ be the partition of unity provided by 
\cite[Proposition 4.9]{RainerSchindl17}, relative to the 
family of cubes $\{Q_i\}_{i \in \N}$ from \cite[Lemma 4.7]{RainerSchindl17},
and let $r_0 = r_0(p)$ be the constant appearing in this proposition. 
The center of $Q_i$ is denoted by $x_i$.
   We claim that an extension of class $\cB^{\{\om\}}$ of $F$ to a suitable neighborhood of $E$ in $\R^n$ is provided by
  \[
  f(x) := 
  \begin{cases}
    \sum_{i\in \N} \vh_{i,p}(x) \, T_{\hat x_i}^{p(x_i) } F(x),  & \text{ if } x \in \R^n \setminus E, \\
    F^0(x), & \text{ if } x \in  E,
  \end{cases}
  \]
  where, given $x \in \R^n \setminus E$,  $\hat x$ is any point in $E$ with $d(x) := d(x,E) = |x-\hat x|$ and  
  \[
    p(x):= \max\{2  \Ga_{\ul s^{2\xi}}(L d(x)) -1,0\}.
  \]
  Here $L$ is a positive constant to be specified below.
  Recall that $Q^*_i$ is the closed cube with the same center as $Q_i$ expanded by the factor $9/8$. 
  By \cite[Corollary 4.8]{RainerSchindl17}, 
  \begin{equation} \label{cor48}
    \frac12 d(x) \le d(x_i) \le 3 d(x), \quad \text{ for all } x \in Q_i^*. 
  \end{equation}
  Then $d(x)<1/(3L\ul s^{2\xi}_1)$ guarantees that both $\Ga_{\ul s^{2\xi}}(L d(x_i))$ and $\Ga_{\ul s^{2\xi}}(L d(x))$ are $\ge 1$, 
  by \eqref{eq:Gamma},
  thus $p(x_i)= 2  \Ga_{\ul s^{2\xi}}(L d(x_i)) -1$ and $p(x)= 2  \Ga_{\ul s^{2\xi}}(L d(x)) -1$.

  $\bullet$ Replace \cite[Lemma 5.2]{RainerSchindl17} by the following lemma. 
  The only difference in the proof is that one uses \eqref{eq:save2} instead of \cite[(5.4)]{RainerSchindl17}.

\begin{lemma}
    \label{proposition9}
There is a constant $C_0 = C_0(n) >1$ such that,
for all Whitney ultrajets $F= (F^\al)_{\al}$ of class $\cB^{\{V^\xi\}}$ that satisfy \eqref{jets1} and \eqref{jets2},
all $L \ge C_0 \rh$, all $x \in \R^n$, and all $\al\in \N^n$, 
\begin{align}
  |(T_{\hat x}^{p(x)} F)^{(\al)}(x)| &\le C (2L)^{|\al|+1} V^\xi_{|\al|},  \label{prop91} 
  \intertext{and, if $|\al| < p(x)$,}
  |(T_{\hat x}^{p(x)}F)^{(\al)}(x)-F^\al(\hat x)| &\le C (2L)^{|\al|+1} |\al|!\, v^\xi_{|\al|+1} d(x).
  \label{prop92}
\end{align}
  \end{lemma}

  We remark that (here and below) by $(T_{\hat x}^{p(x)} F)^{(\al)}(x)$ we mean the $\al$-th partial derivative of the polynomial 
  $y \mapsto T_{\hat x}^{p(x)} F(y)$ evaluated at $y =x$.

  $\bullet$ Replace \cite[Lemma 5.3]{RainerSchindl17} by:

  \begin{lemma} \label{lem:H1}
  There is a constant $C_1 = C_1(n)>0$ such that for all $L>C_1 \rh$, all $\be \in \N^n$, 
  and all $x \in Q_i^*$ with $d(x)<1/(3L\ul s^{2\xi}_1)$, 
  \begin{align} \label{H1}
  |\p^\be (T_{\hat x_i}^{p(x_i) } F - 
  T_{\hat x}^{p(x_i)) } F) (x)| 
  &\le C   
     L^{|\be|+1} \ul S^{2\xi}_{|\be|} \, h_{\ul s^{2\xi}}(L d(x_i)).
\end{align}
\end{lemma}

\begin{proof}
  It suffices to consider $|\be| \le p(x_i) = 2  \Ga_{\ul s^{2\xi}}(L d(x_i)) -1 =: 2q-1$. Let $H_1$ denote the left-hand side of \eqref{H1}.
  By \cite[Lemma 5.1 and Corollary 4.8]{RainerSchindl17} and \eqref{eq:save-1}, 
\begin{align*}
  H_1 
  \le C (2n^2 \rh)^{2q} |\be|! \, v^\xi_{2q}  (6d(x_i))^{2q-|\be|}
  \le
  C  (2n^2 \rh)^{2q} 
  |\be|! \, (\ul s^{2\xi}_{q})^2  (6 d(x_i))^{2q-|\be|}.
\end{align*}
By the definition of $q$,
$h_{\ul s^{2\xi}}(L d(x_i)) = \ul s^{2\xi}_{q} (L d(x_i))^{q} \le 
\ul s^{2\xi}_{k} (L d(x_i))^{k}$ for all  $k$.
Thus  
\begin{align*}
  H_1 
  &\le  C   \Big(\frac{12n^2  \rh}{ L}\Big)^{2q} 
   \,  L^{|\be|}\, |\be|!\,\ul s^{2\xi}_{|\be|} \, h_{\ul s^{2\xi}}(L d(x_i)).
\end{align*}
If $L >  12 n^2   \, \rh$, then \eqref{H1} follows. 
\end{proof}

$\bullet$ Replace \cite[Lemma 5.4]{RainerSchindl17} by:

\begin{lemma} \label{lem:H2}
  There is a constant $C_2 = C_2(n)>0$ such that for all $L>C_2 \rh$, all $\be \in \N^n$, and all $x \in Q_i^*$ 
  with $d(x)<1/(3L\ul s^{2\xi}_1)$, 
  \begin{align} \label{H2}
  |\p^\be (T_{\hat x}^{p(x_i)} F - T_{\hat x}^{p(x) } F) (x)| 
  &\le  C \Big(\frac{3L  }{n}\Big)^{|\be|+1} \ul S^{2\xi}_{|\be|}  h_{\ul s^{2\xi}}(3 L d(x)).   
\end{align}  
\end{lemma}

\begin{proof}
Both $p(x_i)$ and $p(x)$ are majorized by $\Ga_{v^\xi}(Ld(x)/2)$, 
indeed, by \eqref{eq:save2}, \eqref{cor48}, and since $\Ga_{v^\xi}$ is decreasing, 
\[
  p(x_i) = 2  \Ga_{\ul s^{2\xi}}(L d(x_i)) -1 \le 2  \Ga_{\ul s^{2\xi}}(L d(x_i)) = \Ga_{v^\xi}(L d(x_i)) \le \Ga_{v^\xi}(L d(x)/2).
\]
So the degree of the polynomial $T_{\hat x}^{p(x_i) } F - T_{\hat x}^{p(x) } F$
is at most $\Ga_{v^\xi}(Ld(x)/2)$. 
The valuation of the polynomial is equal to $\min\{p(x_i),p(x)\}+1$ (unless $p(x_i) = p(x)$ in which case \eqref{H2} is trivial) and 
so at least $2  \Ga_{\ul s^{2\xi}}(3 L d( x)) =: 2 q$, by \eqref{cor48}. 
So if $H_2$ denotes the left-hand side of \eqref{H2},
then (see the calculation in \cite[(5.7)]{RainerSchindl17})
\begin{align*}
  H_2 
  &\le \frac{C |\be|!}{(n d(x))^{|\be|}} \sum_{j =2 q }^{\Ga_{v^\xi}(Ld(x)/2)}
  (2n^2 \rh  d(x))^{j} v^\xi_{j}. 
\end{align*}
By \eqref{eq:Gamma1}, 
$v^\xi_j (L d(x)/2)^j \le v^\xi_{2 q} (L  d(x)/2)^{2 q}$ for  $2q \le j \le \Ga_{v^\xi}(L  d(x)/2)$. By the definition of $q$,
  $h_{\ul s^{2\xi}}(3L d(x)) = \ul s^{2\xi}_{q} (3L d(x))^{q} \le \ul s^{2\xi}_{k} (3L d(x))^{k}$ for all 
  $k$.
With \eqref{eq:save-1} this leads to 
\begin{align*}
  H_2 
  &\le \frac{C |\be|!}{(n d(x))^{|\be|}} \sum_{j =2 q }^{\Ga_{v^\xi}(L  d(x)/2)}
  \Big(\frac{4 n^2 \rh}{L} \Big)^{j} v^\xi_{2 q} \Big(\frac{L  d(x)}2\Big)^{2 q}
  \\
  &\le \frac{C |\be|!}{(n d(x))^{|\be|}} \sum_{j =2 q }^{\Ga_{v^\xi}(L d(x)/2)}
  \Big(\frac{4n^2 \rh}{L } \Big)^{j} (\ul s^{2\xi}_{q})^2 \Big(\frac{L  d(x)}2\Big)^{2 q}
  \\
  &\le  
  C \Big(\frac{3L}{n}\Big)^{|\be|} |\be|!\, \ul s^{2\xi}_{|\be|}  h_{\ul s^{2\xi}}(3L d(x)) 
  \sum_{j =2 q }^{\Ga_{v^\xi}(L  d(x)/2)}
  \Big(\frac{4n^2  \rh}{L } \Big)^{j}.   
\end{align*}
If we choose $L \ge 8n^2  \rh$, then the sum is bounded by $2$, 
and \eqref{H2} follows.
\end{proof}

$\bullet$ Assume that $L$ is chosen such that 
	\begin{equation} \label{eq:L}
		L > \max\{C_0,C_1,C_2\} \, \rh	
	\end{equation}
	so that \eqref{prop91}, \eqref{prop92}, \eqref{H1}, and \eqref{H2} are valid.
	Recall that $\fW$ denotes the weight matrix associated with $\om$. 
	The next lemma is a substitute for the claim in the proof of Theorem 5.5 in \cite{RainerSchindl17}.
	
\begin{lemma} 
There exist constants $K_j=K_j(n,\om)$, $j = 1,2,3$, such that the following holds.
  If $p = K_1 L$ and $L>K_2 \rh$, then there exist a	
  weight sequence $W \in \fW$ and a constant 
  $M_1= M_1(n,\om,L)>0$ such that 
	for all $x \in \R^n \setminus E$ with $d(x) < \min \{r_0/(3B_1), 1/(3L\ul s^{2\xi}_1)\}$ and all $\al \in \N^n$,
  \begin{equation} \label{eqclaim2}
    |\p^\al (f - T_{\hat x}^{p(x) } F) (x)| \le C M_1^{|\al|+1}  W_{|\al|} h_{\ul s^{4\xi}}(K_3 L d(x)), 
  \end{equation}
  where $C$ and $\rh$ are the constants from \eqref{jets1} and \eqref{jets2} (and $B_1$ is the universal constant from \cite[Lemma 4.7]{RainerSchindl17}).
\end{lemma}

\begin{proof}
	By the Leibniz rule,
\begin{align}
  \p^\al& (f - T_{\hat x}^{p(x) } F) (x) 
  =
  \sum_{\be \le \al} \binom{\al}{\be} \sum_i \vh_{i,p}^{(\al-\be)}(x) \, 
  \p^\be (T_{\hat x_i}^{p(x_i) } F - T_{\hat x}^{p(x) } F) (x). \label{Leibniz}
\end{align}
Now \eqref{H1} and \eqref{H2} imply, that for $x \in Q_i^*$ with $d(x)<1/(3L\ul s^{2\xi}_1)$,
\begin{align} \label{eq:claim1}
  |\p^\be (T_{\hat x_i}^{p(x_i)} F - T_{\hat x}^{p(x)} F) (x)|
  \le C    
    (6 L)^{|\be|+1}  \ul S^{2\xi}_{|\be|} \, h_{\ul s^{2\xi} }(3L d(x)).
\end{align}
As in \cite{RainerSchindl17} we conclude (using \cite[Proposition 4.9]{RainerSchindl17}) that there exist $W = W(p) \in \fW$ and $M = M(p)>0$ such that 
 for all $i \in \N$, all $x \in \R^n \setminus E$ with $d(x) <  r_0/(3B_1)$, and all $\be \in \N^n$,
\begin{equation} \label{eq:claim2}
	|\varphi^{(\be)}_{i,p}(x)|\le  M W_{|\be|}\, \Pi(p,x)	
\end{equation}
where, by \cite[Corollary 3.11]{RainerSchindl17},
\begin{align} \label{eq:Pi}
  \Pi(p,x)&=\exp\Big(\frac{A_1(n)}{p}\si^{\star}\Big(\frac{b_1 p}{9 A_2(n)} d(x)\Big)\Big) \notag
  \\
  &\le \Big(\frac{e}{h_{\ul s^{\et}}(\frac{b_1 p d(x)}{9 A_2(n) B})} \Big)^{\frac{A_1(n) B}{p}}, \quad \text{ for some $B\ge 1$ and all } \et>0. 
\end{align}
($b_1$ is the universal constant from \cite[Lemma 4.7]{RainerSchindl17} and $A_1(n) \le A_2(n)$ are constants depending only on $n$.)
By \eqref{eq:order}, we may assume that 
$\ul S^{2\xi} \le W$. 
Then, by \eqref{Leibniz}, \eqref{eq:claim1}, \eqref{eq:claim2}, and \cite[Lemma 4.7]{RainerSchindl17}, 
for $x \in \R^n \setminus E$ with $d(x) <  \min \{r_0/(3B_1), 1/(3L\ul s^{2\xi}_1)\}$,  
\begin{align*}
  |\p^\al& (f - T_{\hat x}^{p(x) } F) (x)| 
  \\
  &\le
  \sum_{\be \le \al} \frac{\al!}{\be!(\al-\be)!} 
  \cdot 12^{2n} \cdot M W_{|\al|-|\be|}\Pi(p,x)
  \cdot C    (6L)^{|\be|+1}  \ul S^{2\xi}_{|\be|} \, h_{\ul s^{2\xi}}(3L d(x))
  \\
  &\le  12^{2n}C  M
  \Big(\sum_{j=0}^{|\al|} \frac{|\al|!\, n^{|\al|+j}}{j!(|\al|-j)!}  
    (6L)^{j+1} W_{|\al|-j}
        \ul S^{2\xi}_{j}\Big) \, \Pi(p,x)
  \, h_{\ul s^{2\xi}}(3L d(x))    
  \\
  &\le 6\cdot 12^{2n} L C  M n^{|\al|} W_{|\al|}  
  \Big(\sum_{j=0}^{|\al|} \frac{|\al|!\, }{j!(|\al|-j)!}   
      (6 L n)^{j}\Big) \, 
      \Pi(p,x)
  \, h_{\ul s^{2\xi}}(3L d(x))
  \\
  &=  6\cdot 12^{2n} L C  M (n (1 + 6L n ))^{|\al|} W_{|\al|}  
  \Pi(p,x)
  \, h_{\ul s^{2\xi}}(3L d(x)).           
\end{align*}
By \Cref{prop:strongmatrix}, there is $H\ge 1$ (independent of $\xi$) 
such that 
$h_{\ul s^{2\xi}}(t) \le h_{\ul s^{4\xi}}(Ht)^2$ for $t>0$.
Let us choose $L$ according to \eqref{eq:L} and such that $p := 27 \,A_2(n) B H   L/b_1 \ge A_1(n) B$ is an integer. 
Then, by \eqref{eq:Pi} and  since $h_{\ul s^{4\xi}} \le 1$, 
\begin{align*}
	\Pi(p,x) \, h_{\ul s^{2\xi}}(3L d(x))
	\le \frac{e\, h_{\ul s^{2\xi}}(3L d(x))}{h_{\ul s^{4\xi}}(3  H  L    d(x))} 
	\le e\, h_{\ul s^{4\xi}}(3  H  L    d(x))
\end{align*}
 and we obtain \eqref{eqclaim2}. 
 (Note that $M$ depends on $p$, and hence on $L$, which results in the non-explicit dependence of $M_1$.)
\end{proof}

$\bullet$ Let us finish the proof of \Cref{mainadd}. By \eqref{prop91} and 
 \eqref{eqclaim2}, for all $x \in \R^n \setminus E$ with $d(x) <  \min \{r_0/(3B_1), 1/(3L\ul s^{2\xi}_1)\}$ and all $\al \in \N^n$, 
 \begin{align}
    |f^{(\al)}(x)| 
    &\le 
   |(T_{\hat x}^{p(x)} F)^{(\al)}(x)| + |\p^\al (f - T_{\hat x}^{p(x) } F) (x)|
    \label{final}
   \le 
    C M^{|\al|+1} W_{|\al|}
 \end{align}
 for a suitable constant $M=M(n,\om,L)$.

 Let us fix a point $a \in E$ and $\al\in \N^n$. 
 Since $\Ga_{\ul s^{2\xi}}(t) \to \infty$ as $t \to 0$, 
 we have $|\al| < p(x)$ if $x \in \R^n \setminus E$ is 
 sufficiently close to $a$. Thus, as $x \to a$, 
 \begin{align*}
   &|f^{(\al)}(x) - F^{\al}(a)|
   \\
   &\le 
   |\p^\al (f - T_{\hat x}^{p(x) } F) (x)| + 
   |(T_{\hat x}^{p(x)}F)^{(\al)}(x)-F^\al(\hat x)| + |F^\al(\hat x) - F^\al(a)| 
   \\
   & = O(h_{\ul s^{4\xi}}(K_3 L d(x))) + O(d(x)) + O(|\hat x - a|),
 \end{align*}
 by \eqref{jets2}, \eqref{prop92}, and \eqref{eqclaim2}. 
 Hence $f^{(\al)}(x) \to  F^{\al}(a)$ as $x \to a$.
 We may conclude that $f \in C^\infty(\R^n)$ and extends $F$.
 After multiplication with a suitable cut-off function of class $\cB^{\{\om\}}$ with support in 
 $\{x : d(x) <  \min \{r_0/(3B_1), 1/(3L\ul s^{2\xi}_1)\}\}$,
 we find that $f \in \cB^{\{\om\}}(\R^n)$ thanks to \eqref{jets1}, \eqref{final}, and \cite[Lemma 2.4(5)]{RainerSchindl17}.
 The proof of \Cref{mainadd} is complete.

\section{Concave, good, and strong weight functions} \label{sec:questions}

In \cite[Definition 3.5]{RainerSchindl17} we called a weight function $\si$ \emph{good} if its associated weight matrix $\fS$ satisfies 
\eqref{intro:good}. A non-quasianalytic weight function $\om$ is called \emph{strong} if there is a constant $C>0$ such that
\[
\int_1^\infty \frac{\om(tu)}{u^2}\,du \le C \om(t) + C, \quad \text{ for all } t>0. 
\]
Otherwise put, $\om$ is strong if and only if it is equivalent to the concave weight function 
$\kappa=\kappa(\om)$ defined in \eqref{kappa}.
In \cite{RainerSchindl17} we asked the following questions:
\begin{description}
 	\item[Question 3.21] \it Is every concave weight function equivalent to a good one?
 	\item[Question 5.11] \it Is every strong weight function equivalent to a good one?
 \end{description} 
We will give partial answers to these questions and reveal some related connections in \Cref{thm:omegachar} below. 

In \cite{RainerSchindl17} it was important that the associated weight matrix \emph{itself} 
satisfies \eqref{intro:good} as explained after the proof of \Cref{prop:strongmatrix}.  
Since we could overcome this problem (by introducing $\fV = \{V^\xi\}$), it is more natural to allow for a wider concept of goodness.
For completeness we will also treat the Beurling case.
A weight function $\om$ is called \emph{R-good} if there exists a weight matrix $\fM$ satisfying 
\begin{equation} \label{eq:goodR}
 	\A M \in \fM \E N \in \fM \E C\ge 1 \A 1 \le j \le k : \frac{\mu_j}{j} \le C \frac{\nu_k}{k} 
 \end{equation} 
such that $\cB^{\{\om\}} = \cB^{\{\fM\}}$. Recall that $\mu_k := M_k/M_{k-1}$ and $\nu_k := N_k/N_{k-1}$.  
Similarly, $\om$ is called  
\emph{B-good} if there exists a weight matrix $\fM$ satisfying
\begin{equation} \label{eq:goodB}
 	\A N \in \fM \E M \in \fM \E C\ge 1 \A 1 \le j \le k : \frac{\mu_j}{j} \le C \frac{\nu_k}{k} 
 \end{equation} 
 such that $\cB^{(\om)} = \cB^{(\fM)}$.

The next lemma, which is inspired by 
\cite[Proposition 4.15]{Jimenez-GarridoSanz2016}, 
implies that for any weight matrix $\fM$ satisfying \eqref{eq:goodR} (resp.\ \eqref{eq:goodB}) there is a 
weight matrix $\fS$ consisting of strongly log-convex weight sequences such that $\cB^{\{\fM\}} = \cB^{\{\fS\}}$ 
(resp.\ $\cB^{(\fM)} = \cB^{(\fS)}$).

\begin{lemma} \label{lem:good}
  Assume that $1= \mu_0 \le \mu_1 \le \cdots$ and $1= \nu_0 \le \nu_1 \le \cdots$  
  satisfy
  \[
    \E C>0 : \frac{\mu_j}{j} \le C \frac{\nu_k}k, \quad \text{ for all } j \le k.  
  \]
  Then the sequence $\tilde \nu$ defined by 
  \[
     \frac{\tilde \nu_k}k := \inf_{\ell \ge k} \frac{\nu_\ell}\ell, \quad \tilde \nu_0 :=1, 
  \]
  is such that $\tilde \nu_k/k$ is increasing and $C^{-1} \mu \le \tilde \nu \le \nu$.  \qed
\end{lemma}

The next two corollaries are immediate from \Cref{lem:good} 
and results of \cite{RainerSchindl12}, \cite{RainerSchindl14}, and \cite{RainerSchindl16a}.

\begin{corollary}
	Let $\fM$ be a weight matrix
	with the property that for all $M \in \fM$ there is $N \in \fM$ such that $(M_{k+1}/N_k)^{1/k}$ is bounded. 
	Consider the following conditions:
	\begin{enumerate}
		\item[(a)] $\fM$ satisfies \eqref{eq:goodR}.
		\item[(b)] There is a 
		weight matrix $\fS$ consisting of strongly log-convex weight sequences such that $\cB^{\{\fM\}} = \cB^{\{\fS\}}$.
		\item[(c)] $\cB^{\{\fM\}}$ is stable under composition. 
		\item[(d)] $\forall M \in \fM \E N \in \fM \E C>0 \A j\le k : m_j^{1/j} \le C\, n_k^{1/k}$.
	\end{enumerate}
	Then \thetag{a} $\Leftrightarrow$ \thetag{b} $\Rightarrow$ \thetag{c} $\Leftrightarrow$ \thetag{d}. 
	If additionally $\fM$ satisfies 
	\begin{equation} \label{mgR}
    \forall M \in \fM \E N \in \fM  : \mu_k \lesssim N_k^{1/k},
  \end{equation}
	then all four conditions are equivalent.
\end{corollary}

\begin{corollary}
	Let $\fM$ be a weight matrix 
  with the property that for all $N \in \fM$ there is $M \in \fM$ such that $(M_{k+1}/N_k)^{1/k}$ is bounded.
	Consider the following conditions:
	\begin{enumerate}
		\item[(a)] $\fM$ satisfies \eqref{eq:goodB}.
		\item[(b)] There is a 
		weight matrix $\fS$ consisting of strongly log-convex weight sequences such that $\cB^{(\fM)} = \cB^{(\fS)}$.
		\item[(c)] $\cB^{(\fM)}$ is stable under composition. 
		\item[(d)] $\forall N \in \fM \E M \in \fM \E C>0 \A j\le k : m_j^{1/j} \le C\, n_k^{1/k}$.
	\end{enumerate}
	Then \thetag{a} $\Leftrightarrow$ \thetag{b} $\Rightarrow$ \thetag{c} $\Leftrightarrow$ \thetag{d}. 
	If additionally $\fM$ satisfies 
	\begin{equation} \label{mgB}
    \forall N \in \fM \E M \in \fM  : \mu_k \lesssim N_k^{1/k},  
  \end{equation}
	then all four conditions are equivalent.
\end{corollary}

In general, (c) $\not\Rightarrow$ (b) in neither of the corollaries which follows from \cite[Example 3.6]{RainerSchindl12}. 
Note that if $M=N$ then \eqref{mgR} and \eqref{mgB} reduce to a condition which is usually called \emph{moderate growth}  
or $M$.

For weight functions $\om$ we get a full characterization.

\begin{theorem} \label{thm:omegachar}
Let $\om$ be a weight function satisfying $\om(t) = o(t)$ as $t \to \infty$. Then the following are equivalent.
\begin{enumerate}
	\item[(a)] $\om$ is equivalent to a concave weight function. 
	\item[(b)] $\exists C>0 \E t_0 >0 \A \la \ge 1 \A t \ge t_0 : \om(\la t)\le C \la \, \om(t)$.
	\item[(c)] $\cB^{\{\om\}}$ is stable under composition.
  \item[(d)] $\cB^{(\om)}$ is stable under composition.
  \item[(e)] There is a weight matrix $\fS$ consisting of strongly log-convex weight sequences such that
  $\cB^{\{\om\}} = \cB^{\{\fS\}}$.
  \item[(f)] There is a weight matrix $\fS$ consisting of strongly log-convex weight sequences such that
  $\cB^{(\om)} = \cB^{(\fS)}$.
  \item[(g)] $\om$ is R-good.
  \item[(h)] $\om$ is B-good.
\end{enumerate}	
\end{theorem}

Notice that the conditions in the theorem are furthermore equivalent to the classes $\cB^{\{\om\}}$ and $\cB^{(\om)}$ to be 
stable under inverse/implicit functions and solving ODEs,
and, in terms of the associated weight matrix $\fW = \{W^\xi\}_{\xi>0}$, to 
\[
  \A \xi>0 \E \et>0 : (w^\xi_j)^{1/j} \le C\, (w^\et_k)^{1/k} \quad \text{ for } j \le k,
\]
as well as
\[
  \A \et>0 \E \xi>0 : (w^\xi_j)^{1/j} \le C\, (w^\et_k)^{1/k} \quad \text{ for } j \le k,
\]
see \cite{RainerSchindl14}. In the forthcoming paper \cite{FurdosNenningRainer} we shall see that they are also equivalent to 
the property that $\cB^{\{\om\}}$, resp.\ $\cB^{(\om)}$, can be described by almost analytic extensions; see also \cite{PetzscheVogt84}.

\begin{proof}
	The equivalence of the first four conditions \thetag{a}--\thetag{d} is well-known, see e.g.\ \cite{RainerSchindl14}, 
  which is based on \cite[Lemma 1]{Peetre70} and \cite{FernandezGalbis06}.
  That (a) implies (e) and (f) follows from \Cref{prop:strongmatrix}.
	(e) $\Rightarrow$ (c) and (f) $\Rightarrow$ (d) are clear; cf.\ \cite{RainerSchindl12}.
  The equivalences (e) $\Leftrightarrow$ (g) and (f) $\Leftrightarrow$ (h) follow from \Cref{lem:good}.
\end{proof}

\appendix

\section{Weight matrices}

By a \emph{weight matrix} we mean a   
family $\fM$ of weight sequences $M \ge (k!)_k$ which is totally ordered with respect to the pointwise order relation on sequences, i.e., 
\begin{itemize}
	\item $\fM \subseteq \R^\N$,
	\item each $M \in \fM$ is a weight sequence, which means that $M_0 = 1$, $M_k^{1/k} \to \infty$, and $M$ is log-convex,
	\item each $M \in \fM$ satisfies $k! \le M_k$ for all $k$,
	\item for all $M,N \in \fM$ we have $M \le N$ or $M \ge N$. 
\end{itemize} 
For a weight matrix $\fM$ and an open $U \subseteq \R^n$  
we consider the Roumieu class
\begin{align*} 
\cB^{\{\fM\}}(U) &:= \on{ind}_{M \in \fM} \cB^{\{M\}}(U),
\end{align*}
and the Beurling class
\begin{align*} 
	\cB^{(\fM)}(U) &:= \on{proj}_{M \in \fM} \cB^{(M)}(U).
\end{align*} 
For weight matrices $\fM$, $\fN$ we have  
(cf.\ \cite{RainerSchindl12}) 
\begin{align*}
\cB^{\{\fM\}} \subseteq \cB^{\{\fN\}} \quad &\Leftrightarrow  \quad  \A M\in \fM \E N \in \fN : (M_k/N_k)^{1/k} \text{ is bounded},
\\
\cB^{(\fM)} \subseteq \cB^{(\fN)} \quad &\Leftrightarrow  \quad  \A N\in \fN \E M \in \fM : (M_k/N_k)^{1/k} \text{ is bounded}.
\end{align*} 
Analogous equivalences hold for the \emph{local} classes  
\[
	\cE^{\{\fM\}}(U) := \on{proj}_{V \Subset U} \cB^{\{\fM\}}(V) \quad \text{ and } \quad \cE^{(\fM)}(U) := \on{proj}_{V \Subset U} \cB^{(\fM)}(V).    
\]


\def\cprime{$'$}
\providecommand{\bysame}{\leavevmode\hbox to3em{\hrulefill}\thinspace}
\providecommand{\MR}{\relax\ifhmode\unskip\space\fi MR }
\providecommand{\MRhref}[2]{%
  \href{http://www.ams.org/mathscinet-getitem?mr=#1}{#2}
}
\providecommand{\href}[2]{#2}

\end{document}